\documentclass[12pt]{article}

\title{Persistence stability for geometric complexes}
\author{
Fr\'{e}d\'{e}ric Chazal\thanks{\url{frederic.chazal@inria.fr}},
Vin de Silva\thanks{\url{Vin.deSilva@pomona.edu}},
Steve Oudot\thanks{\url{steve.oudot@inria.fr}}}
\date{November 11, 2013}

\usepackage{fullpage}
\usepackage{graphicx}
\usepackage{amsmath,amsthm,amssymb,latexsym}
\usepackage{bbm}
\usepackage{hyperref}

\usepackage{empheq}

\numberwithin{equation}{section}
\theoremstyle{plain}
	\newtheorem{theorem}[equation]{Theorem}
	\newtheorem{lemma}[equation]{Lemma}
	\newtheorem{proposition}[equation]{Proposition}
	\newtheorem{corollary}[equation]{Corollary}
\theoremstyle{definition}
	\newtheorem{definition}[equation]{Definition}
	\newtheorem{example}[equation]{Example}
	\newtheorem*{example*}{Example}
	\newtheorem{remark}[equation]{Remark}

\newcommand{\half}{\tfrac{1}{2}}
\newcommand{\twothirds}{\tfrac{2}{3}}
\newcommand{\quarter}{\tfrac{1}{4}}

\newcommand{\ringfont}{\mathbf}
	\newcommand{\Rr}{{\ringfont{R}}}
	
	\newcommand{\Aa}{\ringfont{A}}
	\newcommand{\kk}{\mathbf{k}}

\newcommand{\Hgr}{\operatorname{H}}

\newcommand{\Hom}{\operatorname{Hom}}

\newcommand{\pmfont}{\mathbb}
	\newcommand{\Ss}{{\pmfont{S}}}
	\newcommand{\Tt}{{\pmfont{T}}}
	\newcommand{\Uu}{{\pmfont{U}}}
	\newcommand{\Vv}{{\pmfont{V}}}
  
\newcommand{\dgm}{\mathsf{dgm}}

\newcommand{\metricfont}{\mathrm}
	\newcommand{\bottle}{\metricfont{d_b}}
	
	\newcommand{\gh}{\metricfont{d_{GH}}}
	\newcommand{\haus}{\metricfont{d_H}}

\newcommand{\rank}{\operatorname{\mathrm{rank}}}

\newcommand{\rips}{\operatorname{\mathrm{Rips}}}
\newcommand{\cech}{\operatorname{\mathrm{\check{C}ech}}}
\newcommand{\wit}{\operatorname{\mathrm{Wit}}}
\newcommand{\dowk}{\operatorname{\mathrm{Dow}}}

\newcommand{\Rips}{\operatorname{\mathbb{R}\mathrm{ips}}}
\newcommand{\Cech}{\operatorname{\mathbb{\check{C}}\mathrm{ech}}}
\newcommand{\Wit}{\operatorname{\mathbb{W}\mathrm{it}}}
\newcommand{\Dowk}{\operatorname{\mathbb{D}\mathrm{ow}}}

\newcommand{\qlen}{1em}
\newcommand{\qem}{\makebox[\qlen]{---}}
  
\newcommand{\qon}[1]{\makebox[\qlen]{$_{\phantom{}}\bullet_{#1}$}}
\newcommand{\qoff}[1]{\makebox[\qlen]{$_{\phantom{}}\circ_{#1}$}}

\newcommand{\e}{\varepsilon}
\newcommand{\one}{\mathbbm{1}}
\newcommand{\dis}{\operatorname{dis}}

\newcommand{\vr}{Vietoris--Rips}
\newcommand{\vc}{\v{C}ech}

\begin{document}

\maketitle

\begin{abstract}
In this paper we study the properties of the homology of different geometric filtered complexes (such as \vr, \vc\ and witness complexes) built on top of totally bounded metric spaces. Using recent developments in the theory of topological persistence,
we provide simple and natural proofs of the stability of the persistent homology of such complexes with respect to the Gromov--Hausdorff distance. We also exhibit a few noteworthy properties of the homology of the Rips and \vc\ complexes built on top of compact spaces.
\end{abstract}

\section{Introduction}
\label{sec:introduction}

The inference of topological properties of metric spaces from approximations is a problem that has attracted special attention in computational topology in recent years.
Given a metric space $(Y,d_Y)$ approximating an
unknown metric space $(X, d_X)$, the aim is to build a simplicial
complex on the vertex set~$Y$ whose homology or homotopy type is the same
as~$X$. Note that, although $Y$ is finite in many applications, finiteness is not a requirement {a~priori}.

Among the many geometric complexes available to us, the \vr\ complex (or simply `Rips complex') is particularly useful, being easy to compute and having good approximation properties. We recall the definition. Let $(X,d_x)$ be a metric space and $\alpha$ a real parameter (the `scale'). Then $\rips(X,\alpha)$ is the simplical complex on~$X$ whose simplices are the finite subsets of~$X$ with diameter at most~$\alpha$:
\[
\sigma = [x_0, x_1, \dots, x_k] \in \rips(X,\alpha)
\; \Leftrightarrow \;
d_X(x_i,x_j) \leq \alpha \;\; \text{for all $i, j$}
\]

When $(X,d_X)$ is a closed Riemannian manifold,
J.-C.~Hausmann \cite{h-ovrcc-95} proved that if $\alpha >0$ is
sufficiently small then the geometric realisation of $\rips(X,\alpha)$
is homotopy equivalent to $X$. This result was later generalised by
J.~Latschev~\cite{l-vrcms-01}, who proved that if $(Y,d_Y)$ is sufficiently close to
$(X,d_X)$ in the Gromov--Hausdorff distance, then there exists
$\alpha>0$ such that $\rips(Y,\alpha)$ is homotopy equivalent to
$X$. Recently, Attali et~al.~\cite{als-vrcap-11} adapted these results
to a class of sufficiently regular compact subsets of euclidean
spaces. For larger classes of compact subsets of Riemannian manifolds,
the homology and homotopy of such sets are known to be encoded in
nested pairs of \vr\ complexes \cite{co-tpbre-08}, yet it
remains still open whether or not a single Rips complex can carry this
topological information.

These approaches make it possible to recover the topology of a metric space $(X,d_X)$ from a sufficiently close approximation $(Y,d_Y)$, provided that the parameter~$a$ is chosen correctly. Unfortunately, this choice very much depends on the geometry of~$X$ and can be difficult (if even possible) to determine in practical applications.
%
%
One way round this issue is to use topological persistence~\cite{elz-tps-02,zc-cph-05}, which encodes the homology of the entire nested family $\Rips(X)=( \rips(X,\alpha) )_{\alpha\in\Rr}$ in a single invariant, the persistence diagram.
Relevant scales~$\alpha$ can then be selected by the user, and
the diagram provides an explicit relationship between the
choice of a scale~$\alpha$ and the homology of the corresponding \vr\
complex. 

The stability of this construction was established by Chazal et
al.~\cite{ccggo-ppmtd-09}, who proved that
\begin{equation}
\tag{*}
\bottle(\dgm (\Rips(X)), \dgm (\Rips(Y)))\leq 2\,\gh(X,Y)
\end{equation}
for finite metric spaces $(X,d_X)$ and $(Y,d_Y)$. Here $\bottle$ and~$\gh$ denote the bottleneck~\cite[Chap. 8]{e-cti-10} and Gromov--Hausdorff distances, respectively. The bound turns out to be tight, which motivates the use of persistence diagrams as discriminative signatures to compare geometric shapes represented as finite metric spaces.

In this paper we show that the same inequality holds for all
totally bounded metric spaces, and can in fact be extended to a larger
class of filtered geometric complexes on such spaces. 
This includes a new family of examples called {Dowker complexes}.
Our analysis adopts a new perspective, guided by recent developments in the theory of topological and algebraic persistence~\cite{cdsgo-sspm-12} which result in simple and natural proofs. Our contributions are the following:
\begin{itemize}
\item Extending the concept of simplicial map between complexes to the
  one of $\e$-simplicial multivalued map between filtered complexes,
  we show that such maps induce canonical $\e$-interleavings between
  the persistent homology modules of these complexes
  --- Section~\ref{sec:multivalued-maps}.
\item Applying this result to correspondences between metric spaces,
  we establish the $\e$-interleaving of the persistent homology modules of
  certain families of filtered geometric complexes (including Rips
  filtrations) built on top of $\e$-close metric spaces --- Section~\ref{sec:correspondences}.
\item We prove the tameness of the persistent homology modules of the
  above filtered complexes when the vertex sets are
  totally bounded. Combined with the previous results, this result shows that
  inequality~(*) can be generalised as claimed above
  --- Section~\ref{subsec:precompact}.
\end{itemize}
In addition to this consistent set of results, the
Section~\ref{subsec:nonpersistent} presents a few noteworthy properties of the homology groups of Rips and \vc\ complexes of totally bounded metric spaces.
We finish in Section~6 with some results on the persistence diagrams of path-metric and $\delta$-hyperbolic metric spaces.

\begin{remark}
Why total boundedness?
Recall that a metric space is totally bounded if for every $\e > 0$ it admits a finite $\e$-sample. In other words, such a space is 
approximable at every resolution by a finite metric space. This explains the good behaviour of these spaces with respect to persistent homology; it is a manifestation of the good behaviour of persistent homology with respect to approximations.
\end{remark}



\section{Persistence modules and persistence diagrams}
\label{sec:persmod}

We adopt the approach and the notation of \cite{cdsgo-sspm-12}. In this section we recall the definitions and results that we need. For a detailed presentation the reader is referred to \cite{cdsgo-sspm-12}.

A {persistence module} $\Vv$ over the real numbers~$\Rr$ is an indexed family of vector spaces\footnote{%
All vector spaces are taken to be over an arbitrary field~$\kk$, fixed throughout this paper.
}
$
(V_a )_{a \in \Rr}
$
together with a doubly-indexed family of linear maps 
$
 \left( v_a^b: V_a \to V_b \mid a \leq b \right)
$
which satisfy the composition law
$
v_b^c \circ v_a^b = v_a^c
$
whenever $a \leq b \leq c$, and where $v_a^a$ is the identity map on $V_a$.

\begin{example}[homology of a filtered complex]
\label{example:filtered-complex-persistence}
This is the standard example, which we use throughout this paper.
Let $\Ss$ be a filtered simplicial complex: that is, a family $(\Ss_a)_{a \in \Rr}$ of subcomplexes of some fixed simplicial complex~$\overline\Ss$, such that $\Ss_a \subseteq \Ss_b$ whenever $a \leq b$.
Let $V_a = \Hgr(\Ss_a)$ be the homology group\footnote{%
We use simplicial homology with coefficients in the field~$\kk$.}
of~$\Ss_a$, and let $v_a^b: \Hgr(\Ss_a) \to \Hgr(\Ss_b)$ be the linear map induced by the inclusion $\Ss_a \hookrightarrow \Ss_b$.
Since, for any $a \leq b \leq c$, the inclusion $\Ss_a \hookrightarrow \Ss_c$ is the composition of the inclusions $\Ss_a \hookrightarrow \Ss_b$ and $\Ss_b \hookrightarrow \Ss_c$, it follows by functoriality that the linear maps satisfy $v_a^c = v_b^c \circ v_a^b$ and the family $( \Hgr(\Ss_a) )_{a \in \Rr}$ is a persistence module. 
\end{example}

Let $\Uu, \Vv$ be persistence modules over~$\Rr$, and let $\e$ be any real number. A {\bf homomorphism of degree~$\e$} is a collection $\Phi$ of linear maps
\[
\left( \phi_a : U_a \to V_{a+\e} \right)_{a \in \Rr}
\]
such that $v_{a+\e}^{b+\e} \circ \phi_a = \phi_b \circ u_a^b$ for all $a \leq b$.
We write
\begin{align*}
\Hom^\e(\Uu, \Vv) &= \{ \text{homomorphisms $\Uu \to \Vv$ of degree~$\e$} \}.
\end{align*}
Composition is defined in the obvious way.
For $\e \geq 0$, the most important degree-$\e$ endomorphism is the shift map
\[
1_\Vv^\e \in \Hom^\e(\Vv, \Vv)
\]
defined to be the collection of maps $( v_a^{a+\e} )$ from the persistence structure on~$\Vv$.
If $\Phi$ is a homomorphism $\Uu \to \Vv$ of any degree, then by definition $\Phi 1_\Uu^\e = 1_\Vv^\e \Phi$ for all $\e \geq 0$.

\begin{example}[continuing Example~\ref{example:filtered-complex-persistence}]
Given $\e >0$, if $f : \overline{\Ss} \to \overline{\Ss}{}'$ is a simplicial map such that $f$ maps $\Ss_a$ to~$\Ss'_{a + \e}$ for any $a\in \Rr$, then $f$ induces a homomorphism of degree~$\e$ between the persistence modules $\Hgr(\Ss)$ and $\Hgr(\Ss')$.
\end{example}

 Two persistence modules $\Uu, \Vv$ are said to be {\bf $\e$-interleaved} if there are maps
\[
\Phi \in \Hom^\e(\Uu, \Vv),
\quad
\Psi \in \Hom^\e(\Vv, \Uu)
\]
such that $\Psi \Phi = 1_\Uu^{2\e}$ and $\Phi \Psi = 1_\Vv^{2\e}$.

Following \cite{ccggo-ppmtd-09,cdsgo-sspm-12} we say that a persistence module $\Vv$ is {\bf q-tame} if
\[
\rank(v_a^b) < \infty
\quad
\text{whenever $a < b$}.
\]
This regularity condition ensures that persistence modules behave well:
\begin{theorem}[\cite{cdsgo-sspm-12}] 
\label{thm:stability2} 
If $\Uu$ is a q-tame module then it has a well-defined persistence diagram $\dgm(\Uu)$.
If $\Uu, \Vv$ are q-tame persistence modules that are $\e$-interleaved then there exists an $\e$-matching between the multisets $\dgm(\Uu)$, $\dgm(\Vv)$. Thus, the bottleneck distance between the diagrams satisfies the bound $\bottle(\dgm (\Uu), \dgm (\Vv)) \leq \e$.
\qed
\end{theorem}

\section{Multivalued maps}
\label{sec:multivalued-maps}

The notion of a simplicial map between simplicial complexes extends to
the notion of an $\e$-simplicial map between filtered simplicial
complexes in the following way:

\begin{definition}
Let $\Ss$ and $\Tt$ be two filtered simplicial complexes with vertex
sets $X$ and $Y$ respectively. A map $f: X \to Y$ is $\e$-simplicial
from $\Ss$ to $\Tt$ if it induces a simplicial map
$\Ss_a\to\Tt_{a+\e}$ for every $a \in \Rr$.
Equivalently, $f$ is $\e$-simplicial if and only if for any $a \in \Rr$ and any simplex $\sigma \in \Ss_a$, $f(\sigma)$ is a simplex of $\Tt_{a+\e}$.
\end{definition}

We wish to extend this concept to multivalued maps. Here are the basic notions.

A \textbf{multivalued map} $C : X \rightrightarrows Y$ from a set $X$ to a set
$Y$ is a subset of $X \times Y$, also denoted $C$, that projects
surjectively onto $X$ through the canonical projection $\pi_X:X \times Y\to
X$. The \textbf{image} $C(\sigma)$ of a subset $\sigma$ of $X$ is the canonical
projection onto $Y$ of the preimage of  $\sigma$ through $\pi_X$.
A (single-valued) map $f$ from~$X$ to~$Y$ is \textbf{subordinate} to~$C$ if we have $(x, f(x)) \in C$ for every $x \in X$; then we write $f: X \stackrel{C}{\to} Y$.
The \textbf{composite} of two multivalued maps $C : X \rightrightarrows Y$ and $D : Y \rightrightarrows Z$ is the multivalued map $D \circ C : X \rightrightarrows Z$, defined by:
\[
(x,z) \in D \circ C 
\; \Leftrightarrow \;
\text{there exists $y \in Y$ such that $(x,y) \in C$ and $(y,z) \in D$}
\]
The \textbf{transpose} of $C$, denoted $C^T$, is the image of $C$ through the symmetry map $(x,y)\mapsto(y,x)$.  Although $C^T$ is well-defined as a subset of $Y\times X$, it is not always a multivalued map because it may not project surjectively onto $Y$.

We now discuss simplicial multivalued maps.

\begin{definition}
Let $\Ss$ and $\Tt$ be two filtered simplicial complexes with vertex sets $X$ and $Y$ respectively. A multivalued map $C : X \rightrightarrows Y$ is \textbf{$\e$-simplicial} from $\Ss$ to $\Tt$ if for any $a \in \Rr$ and any simplex $\sigma \in \Ss_a$, every finite subset of $C(\sigma)$ is a simplex of $\Tt_{a+\e}$.
\end{definition}

\begin{proposition} \label{prop:e-simplicial1}
Let $C : X \rightrightarrows Y$ be an $\e$-simplicial multivalued map from $\Ss$ to
$\Tt$. Then $C$ induces a canonical linear map $\Hgr(C) \in \Hom^\e(\Hgr(\Ss),\Hgr(\Tt))$, equal to $\Hgr(f)$ for any~$f$ subordinate to~$C$.
\end{proposition}

\begin{proof}
Any choice of~$f$ induces a simplicial map $\Ss_a \to \Tt_{a+\e}$ at each~$a \in \Rr$, and these maps commute with the inclusions $\Ss_a \hookrightarrow \Ss_{b}$, $\Tt_{a+\e} \hookrightarrow \Tt_{b+\e}$ for all $a \leq b$. Thus $f$ induces $\Hgr(f) \in \Hom^\e(\Hgr(\Ss),\Hgr(\Tt))$.
Any two subordinate maps $f_1, f_2 : X \stackrel{C}{\to} Y$ induce simplicial maps $\Ss_a \to \Tt_{a+\e}$ which are contiguous. Indeed, for any $\sigma \in \Ss_a$ the two simplices $f_1(\sigma), f_2(\sigma)$ span a simplex of $\Tt_{a+\e}$ since their vertices comprise a finite subset of~$C(\sigma)$.
It follows~\cite[Theorems~12.4 \& 12.5]{Munkres} that $\Hgr(f_1) = \Hgr(f_2)$. Thus the map $\Hgr(C)$ is uniquely defined.
\end{proof}

Another immediate consequence is that the induced homomorphism is
invariant under taking subsets of $C$ that are also
multivalued maps:

\begin{proposition} \label{prop:e-simplicial2}
If $C' \subseteq C : X \rightrightarrows Y$ and $C$ is $\e$-simplicial from $\Ss$ to $\Tt$, then $C'$ is $\e$-simplicial from $\Ss$ to $\Tt$ and $\Hgr(C') = \Hgr(C)$.
\end{proposition}

\begin{proof}
Since $C'$ is a mutivalued map contained in $C$, it is also
$\e$-simplicial, and any map $f: X \to Y$ that is
subordinate to $C'$ is also subordinate to $C$, so we have $\Hgr(C') =
\Hgr(C)$.
\end{proof}

Finally, induced homomorphisms compose in the natural way: 

\begin{proposition} \label{prop:e-simplicial3}
Let $\Ss, \Tt, \Uu$ be filtered complexes with vertex sets $X, Y, Z$ respectively. If%
\begin{align*}
\text{$C : X \rightrightarrows Y$ is a $\e$-simplicial multivalued
  map from $\Ss$ to $\Tt$,} \\ 
\text{$D : Y \rightrightarrows Z$ is
  a $\delta$-simplicial  multivalued map from $\Tt$ to $\Uu$\hfill},
\end{align*}
then the composite $D \circ C : X \rightrightarrows Z$ is a $(\e+\delta)$-simplicial multivalued map from $\Ss$ to $\Uu$, and $\Hgr(D \circ C) = \Hgr(D) \circ \Hgr(C)$.
\end{proposition}

\begin{proof}
$D \circ C$ is $(\e+\delta)$-simplicial as an immediate consequence of the
  definition of $\e$-simplicial multivalued map. Let $f : X
  \stackrel{C}{\rightarrow} Y$ be subordinate to $C$, and let $g : Y
  \stackrel{D}{\rightarrow} Z$ be subordinate to $D$. The composite
  $g \circ f : X \stackrel{D \circ C}{\longrightarrow} Z$ is subordinate
  to $D \circ C$, therefore $\Hgr(D \circ C) = \Hgr(D) \circ \Hgr(C)$.
\end{proof}

\section{Correspondences}
\label{sec:correspondences}

\subsection{Interleaving persistence modules of filtered complexes through correspondences}
\label{subsec:interleaving-correspondence}

\begin{definition}
A multivalued map $C : X \rightrightarrows Y$ is a \textbf{correspondence} if
the canonical projection $C\to Y$ is surjective, or equivalently, if
$C^T$ is also a multivalued map.
\end{definition}

We immediately deduce, if $C$ is a correspondence, that the identity maps $\one_X = \{(x,x): x \in X \}$ and $\one_Y = \{(y,y): y \in Y\}$ satisfy
\[
\one_X \subseteq C^T \circ C,
\qquad
\one_Y \subseteq C \circ C^T.
\]
From this property and propositions \ref{prop:e-simplicial2} and \ref{prop:e-simplicial3}, we deduce the following: 

\begin{proposition} \label{prop:canonical-interleaving}
Let $\Ss$, $\Tt$ be filtered complexes with vertex sets $X$, $Y$ respectively.
If $C : X \rightrightarrows Y$ is a correspondence such that $C$ and $C^T$ are both $\e$-simplicial, then together they induce a canonical $\e$-interleaving between $\Hgr(\Ss)$ and $\Hgr(\Tt)$, the interleaving homomorphisms being $\Hgr(C)$ and $\Hgr(C^T)$.
\qed
\end{proposition}

\subsection{Applications to filtered complexes on metric spaces}
\label{subsec:interleaving-applications}

When $(X, d_X)$ and $(Y, d_Y)$ are metric spaces, the distortion of a
correspondence $C : X \rightrightarrows Y$ is defined as follows:
\[
\dis(C) = \sup \{ |d_X(x,x') - d_Y(y,y')| \,:\, (x,y), (x',y') \in C \}
\] 
The Gromov--Hausdorff distance (\cite{bbi-cmg-01}, Theorem 7.3.25) between $(X, d_X)$ and $(Y, d_Y)$ is 
then defined by taking the infimum of the distortions among all the correspondences between $X$ and $Y$:
\[
\gh(X,Y) = {\half} \inf \{ \dis(C) \,:\, \text{$C$ is a correspondence $X \rightrightarrows Y$} \}
\]
Although $\gh$ is not necessarily finite, it is a distance on the set of isometry classes of compact metric spaces:
(i) it is zero if and only if the spaces are isometric;
(ii) a correspondence and its transpose have the same distortion, so $\gh$ is symmetric;
and (iii) 
the composite of two correspondences $C, C'$ is a correspondence with distortion at most $\dis(C')+\dis(C)$, so $\gh$ satisfies the triangle inequality.

The theme of the next few examples is that low-distortion correspondences give rise to $\e$-simplicial maps on filtered complexes.

\subsubsection{The \vr\ complex}
Let $(X, d_X)$ be a metric space. For $a \in \Rr$ we define a simplicial complex $\rips(X, a)$ on the vertex set~$X$ by the following condition:
\[
[x_0, x_1, \dots, x_k] \in \rips(X, a)
\; \Leftrightarrow \;
\text{$d_X(x_i, x_j) \leq a$ for all $i,j$}
\]
For $a \leq 0$, note that $\rips(X,a)$ consists of the vertex set~$X$ alone.
There is a natural inclusion $\rips(X, a) \subseteq \rips(X, b)$ whenever $a \leq b$. Thus, the simplicial complexes $\rips(X,a)$ together with these inclusion maps define a filtered simplicial complex $\Rips(X)$ on~$X$, the \textbf{Vietoris--Rips complex}.

\begin{lemma}[Vietoris--Rips interleaving]
\label{lemma:rips-interleaving}
Let $(X, d_X)$, $(Y, d_Y)$ be metric spaces. For any $\e > 2 \gh(X,Y)$ the persistence modules $\Hgr(\Rips(X))$ and $\Hgr(\Rips(Y))$ are $\e$-interleaved. 
\end{lemma}

\begin{proof}
Let $C : X \rightrightarrows Y$ be a correspondence with distortion at most $\e$. 

If $\sigma \in \rips(X, a)$ then $d_X(x,x') \leq a$ for all $x,x' \in \sigma$.
Let $\tau \subseteq C(\sigma)$ be any finite subset. For any $y,y' \in \tau$ there exist $x,x' \in \sigma$ such that $y \in C(x)$, $y' \in C(x')$, and therefore:
\[
d_Y(y,y') \leq d_X(x,x') \leq a + \e
\]
It follows that $\tau \in \rips(Y, a+\e)$.

We have shown that $C$ is $\e$-simplicial from $\Rips(X)$ to $\Rips(Y)$. Symetrically, $C^T$ is $\e$-simplicial from $\Rips(Y)$ to $\Rips(X)$.
The result now follows from Proposition~\ref{prop:canonical-interleaving}.
\end{proof}

\subsubsection{The intrinsic \vc\ complex}
Let $(X, d_X)$ be a metric space. For $a \in \Rr$ we define a simplicial complex $\cech(X, a)$ on the vertex set~$X$ by the following condition:
\[
[x_0, x_1, \dots, x_k] \in \cech(X, a)
\; \Leftrightarrow \;
\bigcap_{i=0}^k B(x_i,a) \ne \emptyset
\]
Here $B(x,a) = \{x' \in X: d_X(x,x') \leq a \}$ denotes the closed ball with centre $x \in X$ and radius~$a$.
Any point $\bar{x}$ in the intersection $\bigcap_i B(x_i, a)$ is called an $a$-\textbf{centre} for the simplex $[x_0, \dots, x_k]$.

For $a \leq 0$, note that $\cech(X,a)$ consists of the vertex set~$X$ alone. There is a natural inclusion $\cech(X, a) \subseteq \cech(X, b)$ whenever $a \leq b$. Thus, the simplicial complexes $\cech(X,a)$ together with these inclusion maps define a filtered simplicial complex $\Cech(X)$ on~$X$, the \textbf{(intrinsic) \vc\ complex}\footnote{%
We will usually drop the word `intrinsic' unless we contrasting it with `ambient'.
}.

\begin{lemma}[\vc\ interleaving]
\label{lemma:cech-interleaving}
Let $(X, d_X), (Y, d_Y)$ be metric spaces. For any $\e > 2 \gh(X,Y)$ the persistence modules $\Hgr(\Cech(X))$ and $\Hgr(\Cech(Y))$ are $\e$-interleaved. 
\end{lemma}

\begin{proof}
Let $C : X \rightrightarrows Y$ be a correspondence with distortion at most $\e$. 

Consider $\sigma \in \cech(X, a)$. Let $\bar{x}$ be an $a$-centre for~$\sigma$, so $d_X(\bar{x}, x) \leq a$ for all $x \in \sigma$.
Pick $\bar{y} \in C(\bar{x})$.
Now for any $y \in C(\sigma)$ we have $y \in C(x)$ for some~$x \in \sigma$, and therefore:
\[
d_Y(\bar{y},y) \leq d_X(\bar{x},x) + \e \leq a + \e
\]
Let $\tau \subseteq C(\sigma)$ be any finite subset; then $\bar{y}$ is an $(a+\e)$-centre for~$\tau$ and hence $\tau \in \cech(Y, a+\e)$.

We have shown that $C$ is $\e$-simplicial from $\Cech(X)$ to $\Cech(Y)$. Symetrically, $C^T$ is $\e$-simplicial from $\Cech(Y)$ to $\Cech(X)$.
The result now follows from Proposition~\ref{prop:canonical-interleaving}.
\end{proof}

\subsubsection{Ambient \vc\ complexes and Dowker complexes}

The reader should be aware that there are two distinct uses of the phrase `\vc\ complex'. The first is the intrinsic \vc\ complex described in the preceding subsection, that is constructed from a single metric space.
The second, very commonly used in topological data analysis, is built from a pair or triple of spaces.
These `ambient' \vc\ complexes belong to a much more general family, the Dowker complexes. We consider these now.

Let $X \subseteq \Rr^n$. If $X$ is assumed to be sampled from some unknown object, we can attempt to recover the structure of the object by thickening each point to a closed ball of radius~$a$, say. By the Nerve Lemma~\cite[Section 4.G]{h-at-01}, the homotopy type of the thickened set can be retrieved by constructing the nerve of the collection of balls. This has vertex set~$X$ and a simplex for every finite subset of~$X$ for which the corresponding balls have nonempty intersection in~$\Rr^n$.

We write $\cech(X, \Rr^n; a)$ for this nerve, and $\Cech(X, \Rr^n)$ for the filtered complex obtained by varying~$a$.
The second argument~$\Rr^n$ is often omitted in certain literatures, being regarded as implicit.
For us, however, $\Cech(X)$ refers to the intrinsic \vc\ complex, where a simplex is included only if the corresponding balls meet at a point in~$X$ itself.

In general, an \textbf{ambient \vc\ complex} is defined as follows.
Let $L, W$ be subsets (`landmarks' and `witnesses') of an unnamed metric space. For $a \in \Rr$, consider the complex with vertices~$L$ and simplices determined by:
\[
\sigma \in \cech(L, W; a)
\; \Leftrightarrow \;
\text{$\exists w \in W$ such that $d(w,l) \leq a$ for all~$l \in \sigma$}
\]
The resulting filtered complex is denoted $\Cech(L,W)$.

\begin{example}
The intrinsic \vc\ complex $\Cech(X)$ for a metric space~$X$ is equal to $\Cech(X,X)$, where the ambient space is $X$ itself.
\end{example}

More generally, a \textbf{Dowker complex} is defined as follows.
Let $L,W$ be two sets and let $\Lambda: L \times W \to \Rr$ be any function at all.
For $a \in \Rr$, consider the complex with vertices~$L$ and simplices determined by:
\[
\sigma \in \dowk(\Lambda, a)
\; \Leftrightarrow \;
\text{$\exists w \in W$ such that $\Lambda(l,w) \leq a$ for all~$l \in \sigma$}
\]
The resulting filtered complex is denoted $\Dowk(\Lambda)$.

\begin{example}
The intrinsic \vc\ complex $\Cech(X)$ for a metric space~$X$ is equal to $\Dowk(d_X)$, where $d_X: X \times X \to \Rr$ is the metric.
\end{example}

\begin{example}
The ambient \vc\ complex $\Cech(L,W)$ for a pair of subsets of a metric space is equal to $\Dowk(d|_{L \times W})$, where $d$ is the ambient metric.
\end{example}

\begin{remark}
\label{rem:dowker}
We name this complex in honour of C.~H.\ Dowker~\cite{Dowker_1952}, who compared two simplicial complexes constructed from a binary relation.
Dowker's theorem implies that $\dowk(\Lambda, a)$ and $\dowk(\Lambda^{T}, a)$ have the same homotopy type, where
\[
\Lambda^T : W \times L \to \Rr; \; (w,l) \mapsto \Lambda(l,w)
\]
is the `transpose' of~$\Lambda$ (thus changing the vertex set to~$W$).
This is essentially an instance of the Nerve Lemma, since each complex can be interpreted as the nerve of a suitable covering of the other. The function~$\Lambda$ can be thought of as a filtered binary relation.
Since the Nerve Lemma is functorial (i.e.\ respects maps) \cite[Lemma~3.4]{co-tpbre-08} we obtain the stronger conclusion that $\Dowk(\Lambda)$ and $\Dowk(\Lambda^T)$ have the same filtered homotopy type. It follows that they have the same persistent homology, and therefore, where defined, the same persistence diagrams. We call this phenomenon \textbf{Dowker duality}.
\end{remark}

Two sets of data $(L, W, \Lambda)$ and $(L', W', \Lambda')$ may be compared using a pair of correspondences $C : L \rightrightarrows L'$ and $D : W \rightrightarrows W'$. We define the distortion for such a pair to be:
\[
\dis(C,D) =
\sup_{(l,l') \in C}
\,
\sup_{(w,w') \in D}
\,
|\Lambda(l,w) -\Lambda'(l',w')|
\] 

\begin{lemma}[Dowker interleaving]
\label{lemma:ex-cech-interleaving}
Let $L, L', W, W'$ be sets with functions $\Lambda: L \times W \to \Rr$ and $\Lambda': L' \times W' \to \Rr$. 
If $C: L \rightrightarrows L'$ and $D: W \rightrightarrows W'$ are correspondences and $\e \geq \dis(C,D)$ then the persistence modules $\Hgr(\Dowk(\Lambda))$ and $\Hgr(\Dowk(\Lambda'))$ are $\e$-interleaved.
\end{lemma}

\begin{proof}
Consider $\sigma \in \dowk(\Lambda, a)$. Let $w$ be an $a$-centre for~$\sigma$, so $\Lambda(l,w) \leq a$ for all $l \in \sigma$.
Pick $w' \in C(w)$.
For any $l' \in C(\sigma)$ we have $l' \in C(l)$ for some~$l \in \sigma$, so:
\[
\Lambda'(l',w') \leq \Lambda(l,w) + \e \leq a + \e
\]
It follows that each finite $\sigma' \subseteq C(\sigma)$ belongs to $\dowk(\Lambda',a+\e)$.

We have shown that $C$ is $\e$-simplicial from $\Dowk(\Lambda)$ to $\Dowk(\Lambda')$. Symetrically, $C^T$ is $\e$-simplicial from $\Dowk(\Lambda')$ to $\Dowk(\Lambda)$.
The result now follows from Proposition~\ref{prop:canonical-interleaving}.
\end{proof}

Let $\haus$ denote the Hausdorff distance between subsets of a metric space.

\begin{corollary}[ambient \vc\ interleaving]
\label{cor:clas-cech-interleaving}
Let $L, L'$ and $W$ be subsets of a metric space.
For any $\e > \haus(L,L')$ the ambient \vc\ persistence modules $\Hgr(\Cech(L,W))$ and $\Hgr(\Cech(L',W))$ are $\e$-interleaved.
\end{corollary}

\begin{proof}
We regard the complexes as $\Dowk(\Lambda)$, $\Dowk(\Lambda')$, where $\Lambda = d|_{L \times W}$ and $\Lambda' = d|_{L' \times W}$. Since $\e > \haus(L,L')$ the sets
\begin{align*}
C
&= \{ (l, l') \,:\, \text{$l \in L$, $l' \in L'$ and $d(l,l') < \e$} \},
\\
D &= \{ (w,w) \,:\, w \in W \}
\end{align*}
are correspondences, and $\dis(C,D) \leq \e$.
It follows from Lemma~\ref{lemma:ex-cech-interleaving} that the two persistence modules are $\e$-interleaved.
\end{proof}

It is worth pointing out that Corollary~\ref{cor:clas-cech-interleaving} is well known in the special case $L, L' \subseteq W = \Rr^n$. The usual argument is based on the Nerve Lemma, so it relies on the local topological properties of euclidean space and does not work in general. The elementary proof here shows that the dependence on the Nerve Lemma is  unnecessary.

\subsubsection{The witness complex}

Let $L, W$ be two sets (`landmarks' and `witnesses') and let $\Lambda : L \times W \to \Rr$ be any function. For any finite subset $\sigma \subseteq L$, and any $w \in W$ and $a \in \Rr$, we say that $w$ is an {\bf $a$-witness} for the simplex~$\sigma$ iff
\[
\Lambda(l,w) \leq \Lambda(l',w) + a
\quad
\text{for all $l\in \sigma$ and $l'\in L\setminus\sigma$}.
\]
Given $L, W$ and $\Lambda$, we can then define for any $a \in \Rr$ a simplicial complex $\wit(L,W; a)$ by
\[
\sigma \in \wit(L, W; a)
\;
\Leftrightarrow
\;
\forall \tau \subseteq \sigma,
\;
\exists w \in W
\;
\text{such that $w$ is an $a$-witness for $\tau$.}
\]

There is a natural inclusion $\wit(L,W; a) \subseteq \wit(L,W; b)$ when $a \leq b$, since an $a$-witness is obviously a $b$-witness. The simplicial complexes $\wit(L,W;a)$ together with these inclusion maps define a filtered simplicial complex $\Wit(L,W)$ with vertex set $L$, called the {\bf witness complex filtration}.

\begin{remark}
We alert the reader that the witness complex defined here has nontrivial behaviour for $a < 0$, unlike the \vr\ and \vc\ complexes. This can be suppressed if necessary.
\end{remark}

We now show that the witness complex filtration is stable with respect to varying the witness set while keeping the landmark set fixed.
Let $L$ be a set and let $W, W'$ be witness sets for~$L$ with respect to maps $\Lambda : L \times W \to \Rr$ and $\Lambda' : L \times W' \to \Rr$. The distortion of a correspondence $C : W \rightrightarrows W'$ is defined: 
\[
\dis(C) =
\sup_{l\in L} \ \sup_{(w,w')\in C} \ | \Lambda(l,w) - \Lambda'(l,w') |
\]

\begin{lemma}[witness complex interleaving]
\label{lemma:wit-interleaving}
Let $L$ be a set, and let $W, W'$ be two witness sets of~$L$ with respect to maps $\Lambda : L \times W \to \Rr$ and $\Lambda' : L \times W' \to \Rr$.
If $C: W \rightrightarrows W'$ is a correspondence and $\e \geq 2 \dis(C)$ then the persistence modules $\Hgr(\Wit(L,W))$ and $\Hgr(\Wit(L,W'))$ are $\e$-interleaved.
\end{lemma}

\begin{proof}
Let $\sigma \in \wit(L,W; a)$. For every $\tau \subseteq \sigma$ we argue as follows. Let $w \in W$ be an $a$-witness for~$\tau$, and select $w' \in C(w)$.
For all $l\in \tau$ and $l'\in L\setminus\tau$ we have
\[
\Lambda'(l,w')
\leq \Lambda(l,w) +\half\e
\leq \Lambda(l',w) + a + \half\e
\leq \Lambda'(l',w') + a + \e
\]
so $w' \in W'$ is an $(a+\e)$-witness for~$\tau$.
It follows that $\sigma \in \wit(L, W'; a+\e)$.

Thus, the identity $\one_L$ is an $\e$-simplicial map $\Wit(L,W) \to\Wit(L,W')$, and $\one_L^T = \one_L$ is an $\e$-simplicial map $\Wit(L,W') \to \Wit(L,W)$ by symmetry.
We conclude that $\Hgr(\Wit(L,W))$ and $\Hgr(\Wit(L,W'))$ are $\e$-interleaved.
\end{proof}

The most common form of witness complex takes $L, W$ to be subsets of a metric space with $\Lambda = d|_{L \times W}$ restricted from the ambient metric. Different witness sets may be compared using the Hausdorff distance~$\haus$.

\begin{corollary}
\label{cor:wit-interleaving-metric}
Let $L,W,W'$ be subsets of a metric space, where $W, W'$ are witness sets for~$L$ with respect to $\Lambda = d|_{L \times W}$ and $\Lambda' = d|_{L \times W'}$. For any $\e > 2 \haus(W,W')$ the persistence modules $\Hgr(\Wit(L,W))$ and $\Hgr(\Wit(L,W'))$ are $\e$-interleaved.
\end{corollary}

\begin{proof}
Since $\haus(W,W') < \half \e$, the set
\[
C = \{ (w,w') \in W \times W' : d_X(w,w') < \half\e \}
\]
is a correspondence with $\dis(C) \leq \half\e$. Now apply Lemma~\ref{lemma:wit-interleaving}.
\end{proof}

Unfortunately, in full generality there is no equivalent of Lemma~\ref{lemma:wit-interleaving} in the case where the set $L$ is perturbed, even if the set of witnesses is constrained to stay fixed ($W=W'$). In contrast to ambient \vc\ complexes, the vertices in a witness complex interfere with one another. Here is an explicit counterexample:

\begin{example}
On the real line, consider the sets $W=L=\{0,1\}$ and
$L'=\{-\delta, 0, 1, 1+\delta\}$, where $\delta\in (0,1/2)$ is
arbitrary. Then
\[
\wit(L,W;a)=\{[0],\ [1],\ [0,1]\}
\]
for all $a\geq 0$, whereas
\[
\wit(L',W;a)
=\{[-\delta],\ [0],\ [1],\ [1+\delta],\ [-\delta,0],\ [1,1+\delta]\}
\]
for all $a\in [\delta,1-\delta)$.
Thus, $\Hgr(\Wit(L,W))$ and $\Hgr(\Wit(L',W))$ are not $\e$-interleaved for any $\e<1-2\delta$, whereas $\haus(L,L')=\delta$ can be made arbitrarily small compared to $1-2\delta$.
\end{example}

Note that the set of witnesses in this example is fairly sparse
compared to the set of landmarks. This raises several  interesting
questions, such as whether densifying $W$ (e.g. taking the full real
line) would allow to regain some stability. These questions lie beyond the scope of the paper.

\subsubsection{Generalisation to dissimilarity spaces}
In data analysis one often considers data sets~$X$ equipped with a dissimilarity measure, i.e.\ a map $\tilde d_X : X \times X \to \Rr$ that satisfies $\tilde d_X(x,x) \leq \tilde d_X(x,y) = \tilde d_X(y,x)$ for all $x,y\in X$ but is not required to satisfy any of the other metric space axioms.
It is easily seen that the definitions for \vr, \vc\ and witness complexes continue to make sense for such spaces, and that the distortion of a correspondence $C : X \rightrightarrows Y$ is well-defined. Moreover, since the proofs of our
interleaving results do not make use of any other distance axiom
(triangle inequality, non-negativity, zero property), they remain
valid in this more general context.

\section{Regularity of Rips and \vc\ filtrations}
\label{sec:regularity}

\subsection{Stability of Rips and \vc\ persistence for totally bounded spaces}
\label{subsec:precompact}

The stability theorem for persistent homology is often expressed in terms of persistence diagrams. In this section we show that the \vr\ and \vc\ complexes of a totally bounded metric space have sufficiently tame persistent homology that their persistence diagrams are well defined. There is a similar result for Dowker complexes. The interleaving results of the previous section immediately imply a stability theorem for the persistence diagrams.

We recall the definitions.
Given a positive real number $\e > 0$, a subset $F \subseteq X$ of a metric space $(X,d_X)$ is an \textbf{$\e$-sample} of~$X$ if for any $x \in X$ there exists $f \in F$ such that $d_X(x,f) < \e$. A metric space $(X,d_X)$ is \textbf{totally bounded}  if it has a finite  $\e$-sample for every $\e > 0$.
Bounded subsets of euclidean space are totally bounded. In general a metric space is totally bounded if and only if its completion is compact.

\begin{proposition}
\label{prop:precompact-tame}
If $(X,d_X)$ is a totally bounded metric space then the persistence modules $\Hgr(\Rips(X))$ and $\Hgr(\Cech(X))$ are q-tame. 
\end{proposition}

\begin{proof}
Let us first consider the case of the \vr\ persistence module. We must show that the map $I_a^b: \Hgr(\rips(X, a)) \to \Hgr(\rips(X, b))$ induced by the inclusion has finite rank whenever $a < b$. Let $\e = (b-a)/2$. Since $X$ is totally bounded there exists a finite $\half\e$-sample $F$ of $X$. 
The set $C = \{ (x,f) \in X \times F: d_X(x,f) < \half\e \}$ is an $\e$-corresponence, so the Gromov--Hausdorff distance between $F$ and $X$ is upper-bounded by $\half\e$. It follows from Lemma~\ref{lemma:rips-interleaving} that there exists an $\e$-interleaving between $\Hgr(\Rips(X))$ and $\Hgr(\Rips(F))$.
Using the interleaving maps, $I_a^b$ factorises as
\[
\Hgr(\rips(X,a))
\to \Hgr(\rips(F, a+\e))
\to \Hgr(\rips(X, a+2\e))
= \Hgr(\rips(X, b)).
\]
The second term is finite dimensional since $\rips(F; a+\e)$ is a finite simplicial complex, so $I_a^b$ has finite rank.

The proof for the \vc\ persistence module is the same.
\end{proof}

The above proposition implies that the persistence diagrams of $\Hgr(\Rips(X))$ and~$\Hgr(\Cech(X))$ are well-defined for totally bounded metric spaces. We may now apply the persistence stability theorem to get the following result, which relates the Gromov--Hausdorff distance between two spaces to the bottleneck distance between the persistence diagrams of their \vr\ and \vc\ filtrations. 

\begin{theorem} \label{thm:Rips-Cech-persistence-stability}
Let $X, Y$ be totally bounded metric spaces. Then
\begin{align*}
\bottle(
\dgm (\Hgr(\Rips(X))), 
\dgm (\Hgr(\Rips(Y)))
)
&\leq
2 \gh(X, Y).
\\
\bottle(
\dgm (\Hgr(\Cech(X))), 
\dgm (\Hgr(\Cech(Y)))
)
&\leq
2 \gh(X, Y),
\end{align*}
\end{theorem}

\begin{proof}
This is a consequence of Theorem~\ref{thm:stability2} and Lemmas \ref{lemma:rips-interleaving} and~\ref{lemma:cech-interleaving}.
\end{proof}

\begin{remark}
The first inequality of Theorem~\ref{thm:Rips-Cech-persistence-stability} was earlier proved in~\cite{ccgmo-ghsssp-09} in the special case of finite metric spaces, using a different approach based on embedding the spaces into~$l^\infty$ and invoking the functorial Nerve Lemma.
\end{remark}

Here are the corresponding results for ambient \vc\ complexes.

\begin{proposition}
\label{prop:ambient-tame}
Let $L, W$ be subsets of a metric space. If at least one of $L, W$ is totally bounded, then the ambient \vc\ persistence $\Hgr(\Cech(L,W))$ is q-tame.
\end{proposition}

\begin{proof}
We assume that $L$ is totally bounded.
The case where $W$ is totally bounded follows by Dowker duality (Remark~\ref{rem:dowker}).
It is enough to show for every $\e > 0$ that $\Hgr(\Cech(L,W))$ is $\e$-interleaved with the persistent homology of a finite complex.
To do this, let $F$ be a finite $\e$-sample of~$L$, so that $\haus(L,F) \leq \e$. Then $\Hgr(\Cech(L,W))$ and $\Hgr(\Cech(F,W))$ are $\e$-interleaved by Corollary~\ref{cor:clas-cech-interleaving}, and $\Cech(F,W)$ is finite as required.
\end{proof}

\begin{remark}
The hypothesis in the proposition is most easily checked when the ambient metric space is a `proper' space, meaning that its closed balls are compact. (For instance, $\Rr^n$ is proper.) Then a subset~$L$ is totally bounded if and only if it is bounded.
\end{remark}

Proposition~\ref{prop:ambient-tame} implies that the persistence diagram $\dgm(\Hgr(\Cech(L,W)))$ is well defined when at least one of $L,W$ is totally bounded.

\begin{theorem}
Let $L, L'$ and $W$ be subsets of a metric space. Suppose $L, L'$ are totally bounded, or that $W$ is totally bounded. Then
\[
\bottle(
\dgm(\Hgr(\Cech(L,W))),
\dgm(\Hgr(\Cech(L',W)))
)
\leq
\haus(L,L')
\]
\end{theorem}

\begin{proof}
This is a consequence of Theorem~\ref{thm:stability2} and Corollary~\ref{cor:clas-cech-interleaving}.
\end{proof}

We finish with a tameness result for Dowker complexes.

\begin{proposition}
\label{prop:gen-tame}
Let $L, W$ be sets and $\Lambda: L \times W \to \Rr$ be a function. Suppose the collection $( \lambda_l )_{l \in L}$ of functions
$
\lambda_l(w) = \Lambda(l,w)
$
is bounded and totally bounded with respect to the supremum norm on functions $W \to \Rr$.
Then $\Hgr(\Dowk(\Lambda))$ is q-tame.
\end{proposition}

\begin{remark}
Dowker duality implies that the same conclusion holds if the roles of $L,W$ are interchanged in the hypothesis.
\end{remark}

\begin{proof}
As before, it is enough to show that for any $\e > 0$ the persistence module $\Hgr(\Dowk(\Lambda))$ is $\e$-interleaved with the persistent homology of a finite complex.
To do this, let $F$ be a finite subset of~$L$ such that $(\lambda_l)_{l \in F}$ is an $\e$-sample of $(\lambda_l)_{l \in L}$, and let $\Lambda_F$ be the restriction of $\Lambda$ to $W \times F$.
Then
\begin{align*}
C
&= \{ (l, l') \mid \text{$l \in L$, $l' \in F$ and $\| \lambda_l - \lambda_{l'} \|_\infty < \e$} \}
\\
D &= \{ (w,w) \mid w \in W \}
\end{align*}
are correspondences $L \rightrightarrows F$ and $W \rightrightarrows W$, respectively, with $\dis(C,D) \leq \e$.
It follows from Lemma~\ref{lemma:ex-cech-interleaving} that $\Hgr(\Dowk(\Lambda))$ and $\Hgr(\Dowk(\Lambda_F))$ are $\e$-interleaved, with $\Dowk(\Lambda_F)$ being finite as required.
\end{proof}

One can most straightforwardly use Proposition~\ref{prop:gen-tame} in situations where the Arzel\`{a}--Ascoli theorem guarantees total boundedness. For instance, if $L, W$ are totally bounded metric spaces and $\Lambda$ is Lipschitz then $\Hgr(\Dowk(\Lambda))$ is q-tame.

\subsection{Non-persistent homology of Rips and \vc\ complexes}
\label{subsec:nonpersistent}

The good behaviour of Rips and \vc\ filtrations on compact metric spaces in their {persistent} homology stands in marked contrast to bad behaviour that can be found in the homology groups at particular parameter values. 
Whereas the persistence modules $\Hgr(\Cech(X))$ and $\Hgr(\Rips(X))$ are q-tame when $X$ is a totally bounded metric space, the individual homology groups $\Hgr(\cech(X,a))$ and $\Hgr(\rips(X,a))$ may well be infinite dimensional for some or many values of~$a$.

In the next few sections we present both positive and negative results in this direction.

We briskly remark that homology in dimension zero is easily handled when $X$ is totally bounded: a generating set for both $\Hgr_0(\cech(X,a))$ and $\Hgr_0(\rips(X,a))$ is provided by any $a$-sample of~$X$, so these vector spaces are finite-dimensional when $a > 0$.

\subsubsection{The homology groups of a Rips filtration}
\label{subsub:vr-examples}
It is easy to construct an example of a compact metric space $X$ such that the homology group $\Hgr_1(\rips(X,1))$ has an uncountable infinite dimension. For example consider the union $X$ of two parallel segments in $\Rr^2$ defined by
\[
X = \{ (t,0) \in \Rr^2 : t \in [0,1] \} \cup \{ (t,1) \in \Rr^2 : t \in [0,1] \}
\]
with metric restricted from the euclidean metric in~$\Rr^2$. Then for any $t \in [0,1]$, the edge $e_t = [(t,0), (t,1)]$ belongs to $\rips(X,1)$ but there is no triangle in $\rips(X,1)$ that contains $e_t$ in its boundary. As a consequence, for $t \in (0,1]$ the cycles $\gamma_t = [(0,0), (t,0)] + e_t + [(t,1),(0,1)] - e_0$ are not homologous to $0$ and are linearly independent in $\Hgr_1(\rips(X,1))$.

Here, $a=1$ is the only value of the Rips parameter for which the homology group $\Hgr_1(\rips(X,a))$ fails to be finite-dimensional. In fact, it is possible to construct examples where the set of `bad' values is arbitrarily large.

\begin{proposition} \label{lemma:large-range-infinite-H1-rips}
For any $\alpha, \beta \in \Rr$ such that $0 < \alpha \leq \beta$ and any integer~$k$ there exists a compact metric space~$X$ such that for any $a \in [\alpha,\beta]$, $\Hgr_k(\rips(X,a))$ has an uncountable infinite dimension.
\end{proposition}

\begin{proof}
The following example was obtained with the help of J.-M. Droz who also proved that a similar example can be realised as a subset of $\Rr^4$ endowed with the euclidean metric \cite{d-seslv-12}.

Without loss of generality, we can assume that $\alpha = 1$ and $\beta = 2$. Let us first consider the case $k=1$. Consider the union $X$ of two non-parallel rectangles in~$\Rr^3$, defined as
\begin{eqnarray*}
X =  R_1 \cup R_2 
& = &
\left\{ (t,0,z) \in \Rr^2 : t \in [0,2], z \in [0,1] \right\}
\\
&& \qquad \cup
\left\{ (t,1+\half t,z) \in \Rr^2 : t \in [0,2], z \in [0,1] \right\} 
\end{eqnarray*}
and endowed with the restriction of the $\ell^1$-norm in $\Rr^3$ (see Figure \ref{fig:bigrips_example}). 

\begin{figure}
\centering{
\includegraphics[scale=0.63]{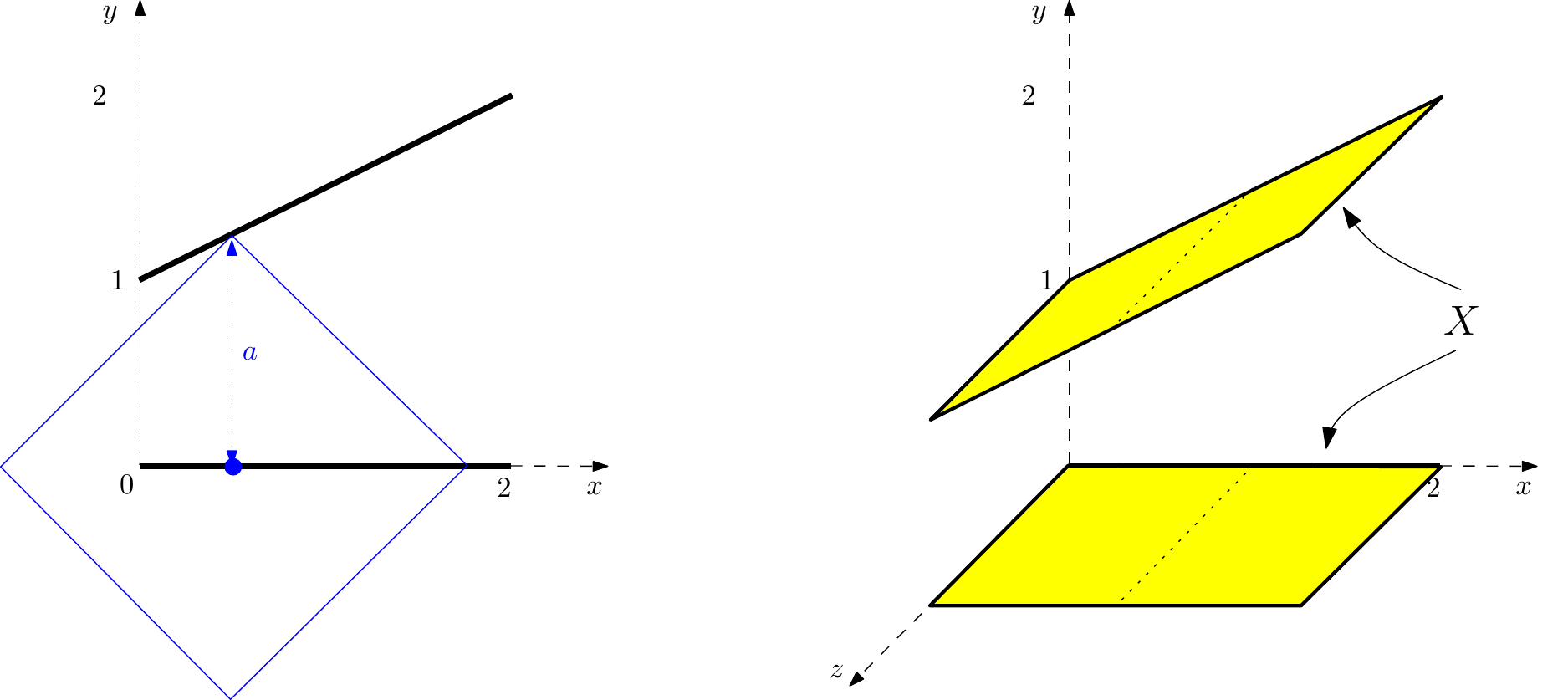}
}
\caption{The union $X$ of the above 2 rectangles endowed with the restriction of the $L^1$ metric in $\Rr^3$ provides an example of a compact metric space such that $\dim \Hgr_1(\rips(X,a)) = \infty$ for any $a \in [1,2]$.}
\label{fig:bigrips_example}
\end{figure}

Since we are using the $\ell^1$-norm, for $a \in [1,2]$ and $z \in [0,1]$, the point $(2(a-1),0,z) \in R_1$ is at distance $a$ from $R_2$ and its unique closest point on $R_2$ is $(2(a-1),a,z)$. As a consequence $e_z = [(2(a-1),0,z),(2(a-1),a,z)]$ is an edge of $\rips(X,a)$ for all $z \in [0,1]$ but there is no (non degenerate) triangle in $\rips(X,a)$ that contains $e_z$ in its boundary. 
Therefore, for $z \in (0,1]$, the cycles $\gamma_z = [(2(a-1),0,0),(2(a-1),0,z)] + e_z + [(2(a-1),a,z),(2(a-1),a,0)] - e_0$ are not homologous to $0$ and are linearly independent in $\Hgr_1(\rips(X,a))$.

To prove the lemma for $k>1$, just consider the product of $X$ with a
$(k-1)$-dimensional sphere of sufficiently large radius (to prevent
the Rips construction from killing the ($k-1$)-homology), and apply the
K\" unneth formula \cite[Theorem 3.16, p.219]{h-at-01}.
\end{proof}

\subsubsection{The open \vr\ filtration}

The examples given above, of \vr\ complexes with infinite-dimensional homology, rely strongly on the fact that $\rips(X,a)$ is defined using a non-strict inequality; that is to say, a closed condition: $[x_0, x_1, \dots, x_k] \in \rips(X, a)$ if and only if $d_X(x_i, x_j) \leq a$ for all $i,j$.

It is natural to ask what happens if strict inequality---an open condition---is used. Given a metric space $(X,d_X)$ and a real number $a \in \Rr$, the open \vr\ complex is the simplicial complex $\rips(X,a^{-})$ with vertex set~$X$ defined by the following condition:
\[
[x_0, x_1, \dots, x_k] \in \rips(X, a^{-})
\; \Leftrightarrow \;
d_X(x_i, x_j) < a,\,
\mbox{for all $i,j$.}
\]

The reader may easily confirm that the examples in section~\ref{subsub:vr-examples} dissolve when the open condition is used. The existence of other constructions is constrained by the following mild regularity result.

\begin{proposition}
\label{prop:open-countable}
For any totally bounded metric space $X$ and real number $a>0$, the total homology $\Hgr(\rips(X,a^{-}))$ has a countable basis. 
\end{proposition} 

\begin{proof}
Any homology class in $\Hgr_k(\rips(X,a^{-}))$ is represented by a cycle, which by definition is a finite linear combination of simplices of diameter strictly less than~$a$ and therefore less than some $a - \tfrac{1}{n}$. It follows that the class lies in the image of $\Hgr_k(\rips(X,a-\frac{1}{n})) \to \Hgr_k(\rips(X,a^{-}))$.
Since $\Hgr_k(\Rips(X))$ is q-tame, by Proposition~\ref{prop:precompact-tame}, this image is finite dimensional.
Since $\Hgr_k(\rips(X,a^{-}))$ is the union of these finite dimensional images for $n \to +\infty$, it has a countable basis.
The full result follows by summing over~$k$.
\end{proof}

We cannot guarantee finite dimensionality. However, Proposition~\ref{prop:open-countable} suggests that we must proceed discretely if we are to find a counterexample.

\begin{proposition}
For any given $a>0$ there exists a totally bounded metric space $X$ such that $\Hgr_1(\rips(X,a^{-}))$ has infinite dimension. 
\end{proposition}

\begin{proof}
We may assume that $a=1$. 
We will construct a bounded subset $X \subset \Rr^2$ whose open \vr\ complex $\rips(X, 1^-)$ has infinite-dimensional 1-dimensional homology. 
We will construct it as the union of two infinite sets $L$ and $R$ such that $\rips(X, 1^-)$ contains the complete graphs on $L$ and $R$, and otherwise each vertex in $L$ shares an edge with precisely one vertex in~$R$, and vice versa. Following the same argument as the one used for the example before Proposition \ref{lemma:large-range-infinite-H1-rips} we will deduce that $\Hgr_1(\rips(X, 1^-))$ is infinite dimensonal. 

We define $X$ in terms of an auxiliary function $f(x) : [0,\infty) \to [0,\infty)$, which will be identified later. We suppose initially that $f$ is continuous, increasing, and positive except at $f(0)=0$.
Specifically, let
\[
L_x = (f(x), x),
\quad
R_x = (1-f(x), x)
\]
for $0 \leq x < f^{-1}(1/2)$. We define
\[
X = \{ L_x, R_x \mid x = \e_n,\, n \gg 0\}
\]
where $(\e_n)$ is a decreasing positive sequence with limit 0.

Clearly $|L_x - R_x| < 1$ for all $x > 0$. We must arrange that $|L_x - R_y| \geq 1$, for $x,y$ distinct elements of the sequence $(\e_n)$. Suppose $x > y$. Then:
\begin{align*}
|L_x - R_y|^2 
&=
  (1 - f(x) - f(y))^2 + (x-y)^2
\\
&\geq
  (1 - 2f(x))^2 + (x-y)^2
\\
&\geq
  1 - 4f(x) + (x-y)^2
\end{align*}
Suppose we have chosen our sequence so that $x > y$ implies $x \geq 2y$; for instance, by setting $(\e_n) = (2^{-n})$. Then $(x-y)^2 \geq \quarter x^2$. Now we choose $f(x) = \frac{1}{16}x^2$.  For $\sqrt{32} > x > y$ in the sequence $(2^{-n})$ we have
\begin{align*}
|L_x - R_y|^2 
\geq
  1 - 4f(x) + (x-y)^2
\geq
  1 - \quarter x^2 + \quarter x^2
= 1
\end{align*}
as required.

It follows that if we define
\[
X = L \cup R
=
\{ (2^{-2n-4}, 2^{-n}) \mid n \geq 1\} \cup \{ (1 - 2^{-2n-4}, 2^{-n}) \mid n \geq 1\}
\]
then $\rips(X, 1^-)$ contains the complete graphs on $L$ and $R$, and otherwise each vertex in $L$ shares an edge with precisely one vertex in~$R$, and vice versa. As a consequence, no triangle in $\rips(X, 1^-)$  contains any of these edges connecting $L$ to $R$. 
For each point $x_n = (2^{-2n-4}, 2^{-n}) \in L$, let $y_n \in R$ be the corresponding point in~$R$, so the edge $e_n = [x_n, y_n]$ is in $\rips(X, 1^-)$. Then, for $n \geq 2$, the cycles $\gamma_n = e_1 + [y_1,y_n]  - e_n + [x_n,x_1]$ and their linear combinations are not homologous to~$0$. So $\dim \Hgr_1(\rips(X, 1^-)) = \infty$.
\end{proof}

\subsubsection{The first homology group of a \vc\ filtration}
Individual \vc\ complexes are almost as badly behaved as individual \vr\ complexes.
It was shown in \cite{bsssw-htms-12} (Appendix B) that the homology groups $\Hgr_k(\cech(X; a))$ of a compact metric space~$X$ can be infinite dimensional for any $k \geq 2$.

However, the first homology is better behaved. The following result was originally obtained by Smale et al.\ \cite[Theorem 8]{bsssw-htms-12} using a different argument.

\begin{proposition}
Let $(X,d)$ be a totally bounded metric space, and let $a \geq 0$. Then, over any coefficient ring~$\Aa$ and any $a \in \Rr$, the 1-dimensional homology $\Hgr_1(\cech(X, a); \Aa)$ is finitely generated over~$\Aa$. In particular, over a field~$\kk$,
\[
\dim_\kk( \Hgr_1(\cech(X,a); \kk) ) < \infty.
\]
\end{proposition}

\begin{proof}
The proposition follows from a sequence of elementary remarks.
\medskip

1. Every $1$-cycle in $\cech(X,a)$ is homologous to a $1$-cycle whose edges have length at most $a$. 

\medskip
{\it Proof}\,
Any edge $[x,x']$ belonging to $\cech(X,a)$ has an $a$-centre; that is, a point~$y$ which satisfies $d(x,y) \leq a$ and $d(x',y) \leq a$. Since $d(y,y) =0 \leq a$, the point $y$ is also an $a$-centre for the triangle $[x,y,x']$ and the edges $[x,y]$, $[y,x']$.
It follows that any 1-cycle
\[
\gamma = \sum_i a_i [x_i, x'_i]
\]
can be replaced by a homologous 1-cycle
\[
\hat\gamma
= \gamma + \partial \sum_i a_i [x_i, y_i, x'_i]
= \sum_i a_i \left( [x_i, y_i] + [y_i,x'_i] \right)
\]
all of whose edges $[x_i, y_i]$, $[y_i, x'_i]$ have length at most~$a$.
{\scriptsize \qed}

\medskip
2. There exists a finite set $E_a$ of edges of length at most~$a$, with the following property: for any edge $[x,y]$ of length at most $a$, there exists an edge $[x',y']$ in~$E_a$ such that $d(x,x') \leq a$ and $d(y,y') \leq a$. 

\medskip
{\it Proof}\,
Since $X$ is totally bounded, so is $X \times X$ with the $\ell^\infty$ product metric
\[
d((x,y), (x',y')) = \max(d(x,x'), d(y,y')).
\]
Since $X \times X$ is totally bounded, so is its subspace
\[
[X \times X]_a
=
\left\{ (x,y) \in X \times X \mid d(x,y) \leq a \right\}.
\]
Let $(x'_1, y'_1), \dots, (x'_N, y'_N)$ be an $a$-sample for
$[X \times X]_a$. Then
\[
E_a = \left\{ 
[x'_1, y'_1], \dots, [x'_N, y'_N]
\right\}
\]
satisfies the required condition.
{\scriptsize \qed}

\medskip
3. Any 1-cycle can be written as a finite linear combination of cycles of the form
\begin{equation}
[x_1,x_2] + [x_2, x_3] + \dots + [x_{k-1}, x_k] + [x_k, x_1]
\end{equation}
($k$ may vary).

\medskip
{\it Proof}\,
This is standard, but we give the proof explicitly.
Certainly any 1-cycle $\gamma$ can be written as a finite linear combination of cycles (as above) and paths of the form
\[
[x_1,x_2] + [x_2, x_3] + \dots + [x_{k-1}, x_k]
+ [x_k, x_{k+1}],
\qquad
x_1 \not= x_{k+1},
\]
(the trivial solution is to use paths of length~1 and no
cycles). Consider the `free' vertices in such a decomposition
for~$\gamma$: that is, vertices that occur as endpoints of the paths
in the decomposition. We can eliminate the free vertices one by one as
follows. Pick a free vertex and enumerate the paths which terminate
there: $P_1, P_2, \dots, P_m$. Since $\partial\gamma = 0$,  we must have
$m \geq 2$. We can decrease $m$ strictly by
concatenating $P_m$ with the appropriate multiple of $P_{m-1}$ or its
reverse. This creates a new, longer path (or cycle, if the other endpoints
coincide) in place of $P_m$, and rescales or annihilates
$P_{m-1}$. Eventually $m=0$ and the free vertex is eliminated.
Finally, when there are no free vertices the decomposition involves only cycles, and we are done.
{\scriptsize \qed}

\medskip
4.
Consider a cycle of the form
\[
\gamma = 
[x_1,x_2] + [x_2, x_3] + \dots + [x_{k-1}, x_k]
+ [x_k, x_1]
\]
whose edges have length at most~$a$. Then $\gamma$ is homologous in $\cech(X, a)$ to a cycle whose edges belong to $E_a$.

\medskip
{\it Proof}\,
Approximate each edge $[x_i, x_{i+1}]$ by an edge $[x'_i, y'_i] \in E_a$ according to remark~2 (interpreting $x_{k+1}$ as $x_1$, cyclically). We claim that
\[
\hat\gamma = 
[x'_1,y'_1] + [y'_1, x'_2] \;+\;
[x'_2,y'_2] + [y'_2, x'_3] \;+
\dots +\;
[x'_k,y'_k] + [y'_{k}, x'_1]
\]
is homologous to $\gamma$ in $\cech(X, a)$. Indeed
\[
\gamma - \hat\gamma
=
\partial
\sum_{i=1}^{k}
\Big(
[x'_i, x_i, y'_i]
+
[y'_i, x_i, x_{i+1}]
+
[y'_i, x_{i+1}, x'_{i+1}]
\Big)
\]
(see figure~\ref{fig:approx_cycle}).
\begin{figure}
\centering{
\includegraphics[scale=0.47]{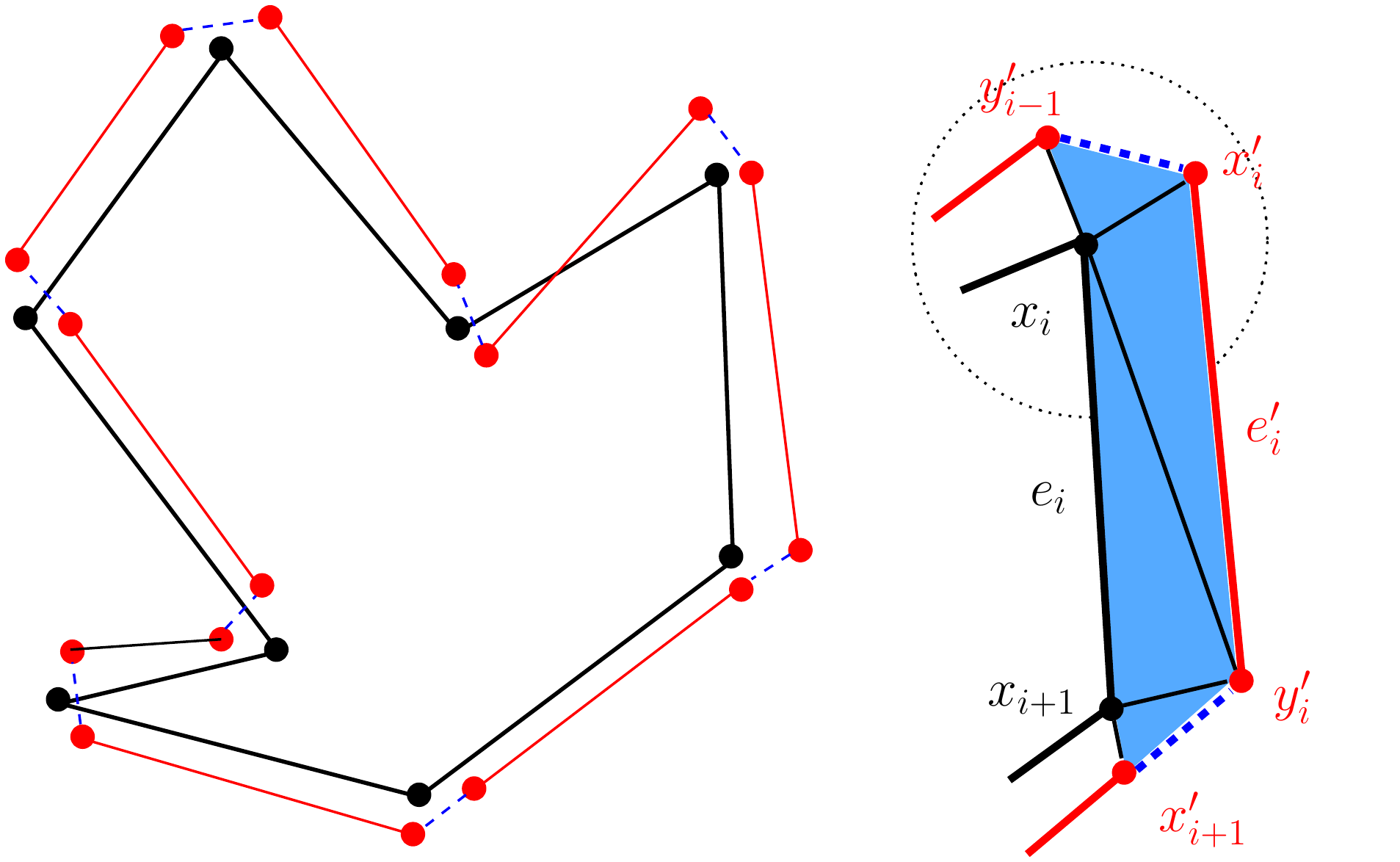}}
\caption{}
\label{fig:approx_cycle}
\end{figure}
To verify that the right-hand side of the equation belongs to $\cech(X, a)$, note that the triangles $[x'_i, x_i, y'_i]$, $[y'_i, x_i, x_{i+1}]$ and $[y'_i, x_{i+1}, x'_{i+1}]$ have $a$-centres $x_i'$, $x_{i+1}$ and $x_{i+1}$ respectively.
{\scriptsize \qed}

\medskip
Combining remarks 1, 3 and~4, we see that in $\cech(X, a)$ any $1$-cycle $\gamma$ is homologous to a 1-cycle involving only edges in the finite set $E_a$. It follows that the homology $\Hgr_1(\cech(X, a); \Aa)$ is finitely generated.
\end{proof}

\section{Special classes of metric spaces}
\label{sec:special-classes}

Up this point we have been considering metric spaces in full generality, occasionally assuming total boundedness. In the last part of this paper, we show that within certain classes of metric space there are constraints on the persistence diagrams of their Rips complexes. The theme is `no new cycles' beyond a certain Rips diameter. For path metric spaces, there are no new 1-cycles once the diameter is positive (Theorem~\ref{thm:H1-Rips-length-space}). For $\delta$-hyperbolic spaces, there are no new 2-cycles after the diameter exceeds~$2\delta$ (Theorem~\ref{thm:H2-Rips-delta-hyperbolic}).
It seems to us that these two results should be the beginning of a much richer story.

\subsection{The persistence diagram of $\Hgr_1(\Rips)$ for a path metric space}

Recall that a metric space $(X,d_X)$ is a path metric space if the distance between each pair of points is equal to the infimum of the lengths of the curves joining these two points.\footnote{See \cite{g-msrnr-07} chap.1, def 1.2 for a definition of the length of a curve in a metric space.}
Every path metric space $(X,d_X)$ satisfies the following property (see \cite{g-msrnr-07}, Theorem 1.8): for any $x, x' \in X$ and any $\e >0$ there exists a point $z \in X$ such that 
\[
\sup ( d_X(x,z) , d_X(x'z) ) \leq \half d_X(x,x') + \e.
\]

\begin{lemma} \label{lemma:H1-Rips-path-metric-space}
Let $(X,d_X)$ be a path metric space and let $q > 0$. Then the map
\[
\Hgr_1(\rips(X,\twothirds q)) \to \Hgr_1(\rips(X,q))
\]
is surjective.
\end{lemma}

\begin{proof}
Any $1$-cycle $\gamma$ in $\rips(X,q)$ can be written as a finite linear combination $\lambda_1 e_1 + \dots + \lambda_n e_n$ of edges $e_i = [x_i, y_i]$ of length at most~$q$. For any~$i$, there exists $m_i \in X$ such that:
\begin{alignat*}{2}
d_X(x_i,m_i) &\leq \half d_X(x_i,y_i) + \e &&\leq \twothirds q
\\
d_X(y_i,m_i) &\leq \half d_X(x_i,y_i) + \e &&\leq \twothirds q
\end{alignat*}
As a consequence the triangle $[x_i, m_i, y_i]$ is contained in $\rips(X,q)$ so $\gamma$ is homologous in $\rips(X,q)$ to
\[
\gamma + \partial \left[  \sum_{i=1}^n \lambda_i [x_i,m_i,y_i] \right]
= \sum_{i=1}^n \lambda_i [x_i,m_i] + \lambda_i [m_i,y_i]
\]
which is a $1$-cycle whose edges belong to $\rips(X, \twothirds q)$.
\end{proof}

\begin{corollary}
\label{cor:H1-Rips-length-space}
Let $(X,d_X)$ be a path metric space and $0 < p < q$. Then the map $\Hgr_1(\rips(X,p)) \to \Hgr_1(\rips(X, q))$ is surjective.
\end{corollary}

\begin{proof}
Iterate the surjectivity of $\Hgr_1(\rips(X, \frac{2}{3}q)) \to \Hgr_1(\rips(X, q))$. Eventually $\left(\frac{2}{3}\right)^n q \leq p$.
\end{proof}

\begin{theorem}
\label{thm:H1-Rips-length-space}
Let $(X,d_X)$ be a path metric space. Then the persistence diagram of $\Hgr_1(\Rips(X))$ is contained in the vertical line $\{0 \} \times [0,+\infty)$.
\end{theorem}

Essentially, this is equivalent to Corollary~\ref{cor:H1-Rips-length-space} in asserting that there are no new 1-cycles once the Rips diameter is strictly positive.
The formal deduction depends on how the persistence diagram is constructed.
Since we are using the measure-theoretic construction of~\cite{cdsgo-sspm-12}, we must argue in terms of the notation and concepts of that paper.
Readers unfamiliar with~\cite{cdsgo-sspm-12} are encouraged to skip the proof and arrive at their own understanding of the relationship between Corollary~\ref{cor:H1-Rips-length-space} and Theorem~\ref{thm:H1-Rips-length-space}.

\begin{proof}
It is enough to show that $\mu(R) = 0$ for any rectangle which does not meet the line. Write $R = [p,q] \times [r,s]$, where $-\infty \leq p < q < r < s \leq +\infty$.
Either $q < 0$, in which case
\[
\mu(R)
=
\langle \qoff{p}\qem\qon{q}\qem\qon{r}\qem\qoff{s} \rangle
\leq \langle \qon{q} \rangle
= \dim \Hgr_1(\rips(X, q)) = 0
\]
because there are no edges when $q < 0$. Or else $0 < p$, in which case
\[
\mu(R)
=
\langle \qoff{p}\qem\qon{q}\qem\qon{r}\qem\qoff{s} \rangle
\leq
\langle \qoff{p}\qem\qon{q} \rangle
= 0
\]
because $\Hgr_1(\rips(X,p)) \to \Hgr_1(\rips(X, q))$ is surjective.
\end{proof}

\begin{remark}
We have not required $X$ to be totally bounded in this theorem.
This would seem necessary for invoking the persistence diagram. In fact, it is shown in~\cite{cdsgo-sspm-12} that the persistence diagram is defined wherever the persistence measure takes finite values. Above we see that the measure is zero away from the vertical line, so the diagram is defined, and empty, away from that line. A similar remark applies to Theorem~\ref{thm:H2-Rips-delta-hyperbolic}.
\end{remark}

Theorem~\ref{thm:H1-Rips-length-space} relies on the fact that any $1$-dimensional simplex in a path metric space can be `subdivided' into a sum of smaller simplices; in other words is homologous to a sum of simplices of strictly smaller diameter. Without further assumptions, this property does not hold for higher-dimensional simplices.
For example if $(X, d_X)$ is a circle of length $1$, then any triple of points $x,y,z$ such that $d_X(x,y) = d_X(y,z) = d_X(z,x) = \frac{1}{3}$ spans a triangle in $\rips(X,\frac{1}{3})$ that is not homologous to any finite sum of triangles of diameter strictly less than $\frac{1}{3}$.

In the next section, we show that there is an analogous result in~$\Hgr_2$ for metric spaces which are $\delta$-hyperbolic.

\subsection{The persistence diagram of $\Hgr_2(\Rips)$ for a $\delta$-hyperbolic space}

\newcommand{\geo}[2]{{{[\![}#1,#2{]\!]}}}

Let $X$ be a geodesic space, that is a metric space where any $x,y \in X$ are connected by geodesic of minimal length $d = d_X(x,y)$. If we choose a minimising geodesic $\geo{x}{y} \subset X$, then $\geo{x}{y}_t$ denotes the point on it which lies at distance~$t$ from~$x$. Taking $\geo{y}{x}$ to be the reverse geodesic, we clearly have $\geo{x}{y}_t = \geo{y}{x}_{d-t}$.

A geodesic space $X$ is said to be {\bf $\delta$-hyperbolic} (see~\cite{gh-90} chap.2) if the sides of every triangle run very close to each other in the following sense. Given $x,y,z \in X$, let $[y,z]$, $[z,x]$, $[x,y]$ be minimising geodesics with lengths $a,b,c$ respectively. The triangle inequality implies that there are non-negative numbers $\alpha, \beta, \gamma$ such that
\[
a = \beta + \gamma,
\quad
b = \gamma + \alpha,
\quad
c = \alpha + \beta,
\]
namely
\[
\alpha = \half(b+c-a),
\quad
\beta = \half (c+a-b),
\quad
\gamma = \half (a+b-c).
\]
The $\delta$-hyperbolicity condition for the triangle is:
\begin{alignat*}{2}
d_X([x,y]_t, [x,z]_t) &\leq \delta
\qquad &&\text{for $0 \leq t \leq \alpha$}
\\
d_X([y,z]_t, [y,x]_t) &\leq \delta
\qquad &&\text{for $0 \leq t \leq \beta$}
\\
d_X([z,x]_t, [z,y]_t) &\leq \delta
\qquad &&\text{for $0 \leq t \leq \gamma$}
\end{alignat*}
If this holds for all triangles $[x,y,z]$, then $X$ is $\delta$-hyperbolic.
\begin{figure}
\hfill
\includegraphics[scale=0.53]{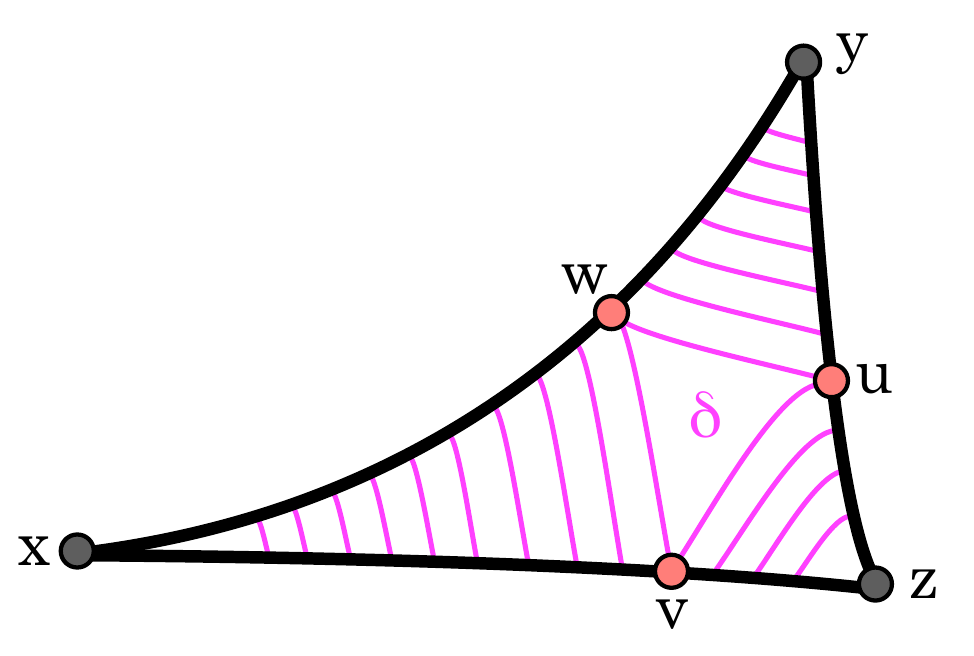}
\hfill
\includegraphics[scale=0.53]{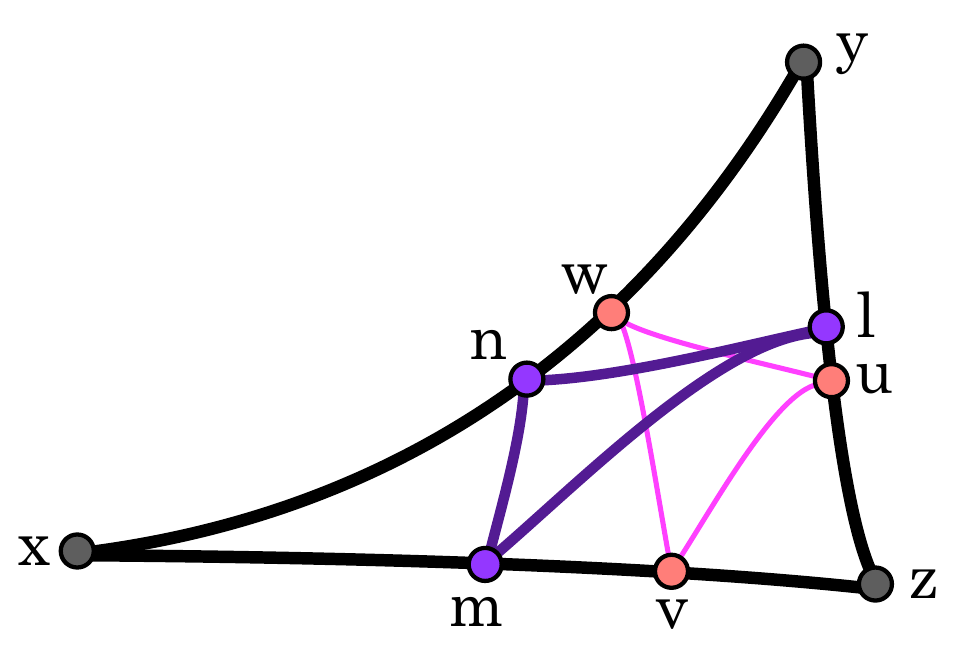}
\hfill{}
\caption{
(left) The triangle $xyz$ is $\delta$-hyperbolic iff the indicated secants have length at most~$\delta$.
(right) The triangle is split into four, with new vertices at the midpoints $\ell, m, n$ of the sides.}
\label{fig:triangle}
\end{figure}
See figure~\ref{fig:triangle} (left). The points
\begin{align*}
u &= \geo{y}{z}_\beta = \geo{z}{y}_\gamma
&
v &= \geo{z}{x}_\gamma = \geo{x}{z}_\alpha
&
w &= \geo{x}{y}_\alpha = \geo{y}{x}_\beta
\end{align*}
have special importance. If $X$ is a tree, then $u,v,w$ coincide. In euclidean or hyperbolic space $u,v,w$ are the points of tangency of the incircle of $xyz$.

\begin{lemma} \label{lemma:H2-Rips-delta-hyperbolic}
Let $(X,d_X)$ be a $\delta$-hyperbolic geodesic space and let $q > 2\delta$. Then the map
$
\Hgr_2(\rips(X,\half q + \delta)) \to \Hgr_2(\rips(Z,q))
$
is surjective.
\end{lemma}

\begin{proof}
Any class $\sigma \in \Hgr_2(\rips(X,q))$ is represented by a 2-cycle which is a linear combination of triangles in $\rips(X,q)$. We must break these triangles into smaller triangles.

We begin by selecting a geodesic $\geo{x}{y}$ for each pair $x,y$ which occurs as an edge of a triangle in~$\sigma$.

Let $T = [x,y,z]$ be a triangle of $\sigma$, with side-lengths $a,b,c$ and $\alpha, \beta, \gamma$ defined as before.
We will break up the triangle using the midpoints
\begin{align*}
\ell &= \geo{y}{z}_{a/2}
&
m &= \geo{z}{x}_{b/2}
&
n &= \geo{x}{y}_{c/2}
\end{align*}
of the sides. To help estimate the new edges we consider the `incircle' points $u,v,w$ defined above. See figure~\ref{fig:triangle} (right).

Along the three geodesic sides of the triangle, one easily calculates:
\begin{align*}
d_X(u, \ell) & = \half | b - c |
&
d_X(v, m) & = \half | c - a |
&
d_X(w, n) & = \half | a - b |
\end{align*}
We estimate the distances between $\ell, m, n$. Since $[x,y,z] \in \rips(X,q)$ we have $\max(a,b,c) \leq q$. 
Moreover, we suppose that $a \leq b \leq c$.
Then:
\begin{align*}
d_X(\ell,m)
&\leq d_X(\ell,u) + d_X(u,v) + d_X(v,m)&&
\\
&\leq \half(c-b) + \delta + \half (c-a)
\\
&\leq \half c - \gamma + \delta
\\
&\leq \half q + \delta
\end{align*}
and:
\begin{align*}
d_X(\ell,n)
&\leq d_X(\ell,u) + d_X(u,w) + d_X(w,n)
\\
&\leq \half(c-b) + \delta + \half (b-a)
\\
&\leq \half c - \half a + \delta
\\
&\leq \half q + \delta
\end{align*}
For the third edge we introduce $m' = \geo{x}{y}_{b/2}$. Since $\half b \leq \alpha$ we can invoke the $\delta$-hyperbolicity condition for $m, m'$ to get:
\begin{align*}
d_X(m,n)
&\leq
d_X(m,m') + d_X(m',n)
\\
&\leq
\delta + \half(c-b)
\\
&\leq
\half q + \delta
\end{align*}
From this, we see that if we replace each triangle $[x,y,z]$ of~$\sigma$ with the corresponding sum
\[
-[x,m,n] + [y,\ell, n] - [z,\ell,m] + [\ell, m, n]
\]
we get a 2-cycle $\hat\sigma$ whose triangles belong to $\rips(X, \half q + \delta)$.

We need to do one more thing, which is to show that $\sigma, \hat\sigma$ are homologous through a 3-cycle whose tetrahedra belong to $\rips(X, q)$. We require an additional edge to do this, the median~$zn$:
\begin{align*}
d_X(z,n)
&\leq d_X(z,u) + d_X(u,w) + d_X(w,n)
\\
&\leq \half(b+a-c) + \delta + \half(b-a)
\\
&\leq \half c + \delta
\\
&\leq q
\end{align*}
Now define
\[
H[x,y,z] =
[x,y,z,n] + [x,z,m,n] - [y,z,\ell,n] + [z, \ell, m, n]
\]
for each triangle $[x,y,z]$, and extend linearly over all triangles in~$\sigma$. One can check that $\hat \sigma = \sigma + \partial H \sigma$ when $\sigma$ is a cycle.
 All edges of the tetrahedra in~$H\sigma$ have length at most~$q$, so $[\sigma] = [\hat \sigma]$ in $\Hgr_2(\rips(X,q))$.
\end{proof}

Iterating the lemma, we get:

\begin{corollary}
\label{cor:H2-Rips-delta-hyperbolic}
Let $(X,d_X)$ be a $\delta$-hyperbolic geodesic space and let $q > p > 2\delta$. Then the map $\Hgr_2(\rips(X,p)) \to \Hgr_2(\rips(X,q))$ is surjective.
\end{corollary}

\begin{proof}
If $f(t) = \half t + \delta$ and $q > 2\delta$ then the iterates $f^n(q)$ are a decreasing sequence converging to~$2\delta$, so eventually $f^n(q) < p$.
\end{proof}

\begin{theorem}
\label{thm:H2-Rips-delta-hyperbolic}
Let $(X,d_X)$ be a $\delta$-hyperbolic geodesic space. Then the persistence diagram of $\Hgr_2(\Rips(X))$ is confined to the vertical strip $[0, 2\delta] \times [0,+\infty)$.
\end{theorem}

\begin{proof}
It is enough to show that $\mu(R) = 0$ for any rectangle which does not meet the strip. Write $R = [p,q] \times [r,s]$, where $-\infty \leq p < q < r < s \leq +\infty$.
Either $q < 0$, in which case
\[
\mu(R)
=
\langle \qoff{p}\qem\qon{q}\qem\qon{r}\qem\qoff{s} \rangle
\leq \langle \qon{q} \rangle
= \dim \Hgr_2(\rips(X, q)) = 0
\]
because there are no triangles when $q < 0$. Or else $2\delta < p$, in which case
\[
\mu(R)
=
\langle \qoff{p}\qem\qon{q}\qem\qon{r}\qem\qoff{s} \rangle
\leq
\langle \qoff{p}\qem\qon{q} \rangle
= 0
\]
because $\Hgr_2(\rips(X,p)) \to \Hgr_2(\rips(X, q))$ is surjective.
\end{proof}

As with Theorem~\ref{thm:H1-Rips-length-space}, the result is essentially equivalent to the corollary that precedes it: there are no new $\Hgr_2$~classes once the Rips diameter exceeds~$2\delta$ so the persistence diagram is empty outside the vertical strip.
Our formal proof is written in the language of~\cite{cdsgo-sspm-12}, and again the reader unfamiliar with the concepts may safely skip it, and simply regard Theorem~\ref{thm:H2-Rips-delta-hyperbolic} as another way to state Corollary~\ref{cor:H2-Rips-delta-hyperbolic}.

\subsection*{Acknowledgements}
The authors thank Steve Smale for fruitful discussions that motivated the results Section~\ref{subsec:nonpersistent},
and J.-M. Droz for suggesting the idea of the proof of Proposition \ref{lemma:large-range-infinite-H1-rips}. 

The authors gratefully acknowledge the following funding sources for this work: Digit{e}o project C3TTA (including the Digit{e}o chair held by the second author); European project CG-Learning (EC contract No.~255827); ANR project GIGA (ANR-09-BLAN-0331-01); DARPA  project Sensor Topology and Minimal Planning `SToMP' (HR0011-07-1-0002). The second author is a 2013 Simons Fellow, and is supported in part by the Institute for Mathematics and its Applications with funds provided by the National Science Foundation.


\bibliographystyle{plain}
\bibliography{RipsBiblio}

\end{document}